\documentclass[10pt]{article}
\usepackage[english]{babel}
\usepackage[utf8]{inputenc}
\usepackage[OT1]{fontenc}
\usepackage{amsfonts, amsmath, amsthm, amssymb}
\usepackage{graphicx}
\usepackage{listings}
\usepackage{hyperref}
\usepackage[margin=1in]{geometry}
\usepackage{xcolor}
\usepackage{algorithm}
\usepackage{algorithmic}

\newtheorem{theorem}{Theorem}[section]
\newtheorem{corollary}{Corollary}[theorem]
\newtheorem{lemma}[theorem]{Lemma}
\theoremstyle{remark}
\newtheorem*{remark}{Remark}

\theoremstyle{definition}
\newtheorem{definition}{Definition}[section]
\newcounter{countCode}
\lstnewenvironment{code} [1][caption=Ponme caption, label=default]{%
	 
	\setcounter{lstlisting}{\value{countCode}} 
	\lstset{ %
	language=python,
	basicstyle=\ttfamily\footnotesize,       
	numbers=left,                   
	numberstyle=\sc,      
	stepnumber=1,                   
	numbersep=5pt,                 
	numberstyle=\color{red!50!blue},
    	backgroundcolor=\color{white},
	rulecolor=\color{white},
	keywordstyle=\color{red}\bfseries,
	showspaces=false,               
	showstringspaces=false,         
	showtabs=false,                 
	frame=single,                   
	framexleftmargin=0mm,
	numberblanklines=false,
	xleftmargin=5pt,
	breaklines=true,
	breakatwhitespace=true,
	breakautoindent=true,
	captionpos=t,
	texcl=true,
	tabsize=2,                     
	extendedchars=true,
	inputencoding=utf8, 
	escapechar=\%,
	morekeywords={print, println, size, background, strokeWeight, fill, line, rect, ellipse, triangle, arc, save, PI, HALF_PI, QUARTER_PI, TAU, TWO_PI, width, height,},
	emph=[1]{print,println,}, emphstyle=[1]{\color{blue}}, 
	emph=[2]{width,height,}, emphstyle=[2]{\bf\color{violet}}, 
	emph=[3]{PI, HALF_PI, QUARTER_PI, TAU, TWO_PI}, emphstyle=[3]\color{orange!50!violet},
	emph=[4]{line, rect, ellipse, triangle, arc,}, emphstyle=[4]\color{green!70!black},
	#1}}{\addtocounter{countCode}{1}}

\title{String representation of trivalent 2-stratifolds with trivial fundamental group}
\author{Myriam Hernández-Ketchul \hspace{1cm} Jesús Rodríguez-Viorato}
\date{June 22nd, 2020}
\begin{document}
\maketitle \tableofcontents 
\begin{abstract}
We give a Python program that is capable to compute and print all the distinct trivalent 2-stratifold graphs up to $N$ white vertices with trivial fundamental group (see \cite{art:class}). Our algorithm uses the three basic operations described in \cite{art:Models} to construct new graphs from any set of given graphs. We iterate this process to construct all the desired graphs.  The algorithm includes an optimization that reduces the repetition of generated graphs, this is done by recognizing equivalent white vertices of a graph under automorphism. We use a variation of the AHU algorithm to identify those vertices and as well to distinguish isomorphic graphs in linear time. The returned string from the AHU algorithm is also used as a hashing function to search for repeated graphs in amortized constant time.

\end{abstract}

\section{Introduction}

Since 2016, the authors José Carlos Gómez-Larrañaga, Francisco González-Acuña, and Wolfgang Heil have studied 2-stratifolds \cite{art:2-stratifold}. There are multiple motivations to follow their work, but one of the most interesting is the applications of the 2-stratifolds on the field of Topological Data Analysis.\\ 

The classification of 2-manifolds has been well studied, but the study of the 2-stratifolds has just started. On \cite{art:Models}, Gómez-Larrañaga, González-Acuña, and Wolfgang started to analyze the ones with trivial fundamental group, giving a process to build them. And continuing with that work, it's easy to ask how many of those exist. 

Moreover, using the operations described in their work, we started to wonder if there was possible to use them to get every 2-stratifold. Yair Hernández, wrote a computational algorithm \cite{StratModule} to build some of them, and based on that code, we decided to extend it to get the 2-stratifolds meticulously; in an optimized way. \\

Since the operations described on \cite{art:Models} are performed on white vertices, it was decided to use the number of white vertices as a parameter for the algorithm, limiting the number of graphs that it was going to build. Therefore our program can calculate and draw the graph representation defined in \cite{art:Models} of every 2-stratifold with trivial fundamental group up to $N$ white vertices.

The code was written on Python and can be consulted on \cite{cod:TSR}. In order to identify different graphs, it used a modified version of the AHU algorithm. And to optimize the algorithm, there is defined an identification among white vertices under automorphisms of the graph (see def. \ref{def:symmetry}). The modified version of the AHU algorithm, the algorithm \ref{AHU_modified},  allows us to identify every graph with a unique string, which is the string representation, that is used as a hash to discard the repeated graphs and optimize the building process.\\

The contents of this paper include a description of the algorithm previously mentioned and the proof that it actually builds all the graphs that it is asked to. In section 2, there are given basic graph theory definitions, that are necessary to understand the proves. In section 3, there are proven some known results of graph isomorphisms but applied to the graph representations of the 2-stratifolds. In the next section,``Characterizing weighted trees with a string", we describe the modified version of the AHU algorithm and we explain how it is helpful to solve the problem of classifying the trivalent graphs and therefore the 2-stratifolds with trivial fundamental group. In the section  ``The graph generator algorithm", we describe the main algorithm that assigns a string representation to every graph, and we give the results of the construction of the graphs up to 11 white vertices, including a comparison between the optimized and non-optimized version. Finally, in ``Nomenclature" we explained the tag that is given to every graph at the end of the algorithm because although the string representation is unique, it doesn't work as a quick resume of the general shape of the graph.

\section{Preamble}
    \begin{definition}\label{def:trivalent_graph}
     We will say that a graph generated by a simply-connected trivalent 2-dimensional stratifold is a \textbf{trivalent graph}.
    \end{definition}

    \begin{definition}\label{def:path}
    Let be $G$ a graph, a set of vertices $\{v_0, v_1, ..., v_n\} \subset G$ such that $v_i$ is adjacent to $v_{i+1}$, and $v_i \neq v_j$ for every $0\leq i, j \leq n-1$ is called a \textbf{path} from $v_0$ to $v_n$ in $G$.
    \end{definition}

    \begin{definition}\label{def:degree_leaf}
    The \textbf{degree of a vertex} $v$ in a graph $G$ is the number of vertices in $G$ that are adjacent to $v$. The set of all the vertices that are adjacent to $v$ are the \textbf{neighbors of $v$} and we denote them by $N(v)$. A vertex of degree 1 is a \textbf{leaf}.
    \end{definition}
        
    \begin{definition}\label{def:length_of_path}
    Given a $u-v$ path $R$ in $G$ the \textbf{length} of the path is the sum of the weights of the edges encountered in $R$. For two vertices $u$ and $v$, the \textbf{shortest} path is the $u-v$ path with minimum length, and the \textbf{largest} path is the $u-v$ path with maximum length. The \textbf{length (without weights)} of a path is the length of the path considering that all the weights are 1.
    \end{definition}
    
    \begin{definition}\label{def:central_vertex_center}
    Let $v$ be a vertex of a graph $G$, its \textbf{eccentricity} $e(v)$ is the length (without weights) of the largest path from $v$ to another vertex in $G$. The \textbf{radius} of $G$, $rad(G)$, is the smallest eccentricity among the vertices of $G$. For any vertex $v$, such that $e(v) = rad(G)$ we say that $v$ is a \textbf{central vertex} of $G$, if $G$ has only one central vertex $u$, we say that $u$ is the \textbf{center} of $G$. 
    \end{definition}
    
    \begin{definition}\label{def:tree_rooted_tree}
    Given a graph $G$, we say that $G$ is a \textbf{tree} if and only if for every $v, w$ vertices of $G$, there exists a path from $v$ to $w$ and there is no path with positive length from $v$ to itself. A \textbf{rooted tree} is a tree with a special vertex identified as the \textbf{root}. 
    \end{definition}
    
    \begin{definition}\label{def:child_vertex}
    On a rooted tree, a vertex $v$ is a \textbf{child} of vertex $w$ if, in the path from the root to $v$, $v$ immediately succeeds $w$. We said that $w$ is \textbf{parent} of $v$ if and only if $v$ is child of $w$.
    \end{definition}

\section{Isomorphic graphs}\label{sec:isomorphism}
    
    \begin{definition}\label{def:iso}
    Two weighted graphs $G$ and $H$ are \textbf{isomorphic}  if there exists a bijective function $\phi: V(G) \to V(H)$ such that two vertices $u$ and $v$ are adjacent in $G$ if and only if $\phi(u)$ and $\phi(v)$ are adjacent in $H$, and for every edge, $u-v$ in $G$, the edge $\phi(u)-\phi(v)$ in $H$ has the same weight as $u-v$. If there is no such function $\phi$ as described above, the $G$ and $H$ are \textbf{non-isomorphic graphs}. (Definition from \cite{book:GraphT}) Moreover, for $G$ and $H$ rooted trees with roots $r_G, r_H$, respectively, we say that $G$ and $H$ are \textbf{isomorphic as rooted trees} if they are isomorphic and $\phi(r_G) = r_H$.
    \end{definition}
    
    \begin{lemma}\label{lemma:leaves_if_iso}
    If two graphs $G$ and $H$ are isomorphic with a function $\phi: V(G) \to V(H)$, for $u$ vertex of $G$, $u$ is a leaf of $G$ if and only if $\phi(u)$ is a leaf of $H$.
    \end{lemma}
    \begin{proof}
    Let $u$ a leaf of $G$, then by definition, there exists only one vertex adjacent to $u$ in $G$, suppose $v$. As $u$ and $v$ are adjacent in $G$, $\phi(u)$ and $\phi(v)$ are adjacent in $H$. For any other $w$ vertex in $G$, different from $u$ or $v$; $w$ is not adjacent to $u$ in $G$ and therefore $\phi(w)$ is not adjacent to $\phi(u)$ in $H$. Then, $\phi(v)$ is the only vertex adjacent to $\phi(u)$ in $H$, concluding that $\phi(u)$ must be a leaf of $H$. Analogously if $\phi(u)$ is a leaf of $H$ we can conclude that $u$ must be a leaf of $G$.
    \end{proof}
    \begin{corollary}\label{cor:num_leaves}
    If two graphs $G$ and $H$ are isomorphic, then they have the same number of leaves.
    \end{corollary}

    \begin{definition}\label{def:iso_tg}
    Two trivalent graphs $G$ and $H$ are \textbf{isomorphic as trivalent graphs} if there is an  isomorphism $\phi$ as weighted graphs such that, if $B(G), B(H)$ are the sets of black vertices of $G, H$, and $W(G), W(H)$ are the sets of white vertices of $G, H$, the functions $\phi\restriction_{B(G)}: B(G) \to B(H)$ and $\phi\restriction_{W(G)}: W(G) \to W(H)$ are bijective.
    \end{definition}
    
    \begin{lemma}\label{lemma:num_vertex}
    If two trivalent graphs $G$ and $H$ are isomorphic, then they have the same number of white and black vertices and the same number of leaves.
    \end{lemma}
    \begin{proof}
    If $G$ and $H$ are two trivalent graphs that are isomorphic, there exists $\phi: V(G) \to V(H)$ such that $\phi\restriction_{B(G)}: B(G) \to B(H)$ and $\phi\restriction_{W(G)}: W(G) \to W(H)$  are bijective, therefore $|B(G)| = |B(H)|$ and $|W(G)| = |W(H)|$. And by corollary \ref{cor:num_leaves},  $G$ and $H$ have the same number of leaves.
    \end{proof}
    
\begin{theorem}[Theorem 2.7 in \cite{book:GraphT}]\label{thm:unique_path}
There exists a unique path between any two vertices of a tree.
\end{theorem}

    \begin{lemma}\label{lemma:same_length_paths}
    For $G$ and $H$ two isomorphic trees with isomorphism function $\phi$, for every two vertices $u$, $v$ in $G$, the length of the $u-v$ path in $G$ is the same as the length of the $\phi(u)-\phi(v)$ path in $H$.
    \end{lemma}
    \begin{proof}
    Given that $G$ and $H$ are trees, by theorem \ref{thm:unique_path} there exists a unique path between any two vertices in them. 
    Let $u$, $v$ be vertices in $G$, there exists $R = \{u = r_0, r_1, ..., r_{n-1}, v = r_n\}$ a unique path between them, such that $R$ is a sequence of adjacent vertices without repetition in $G$. For two adjacent vertices $a, b$ let's denote $(a,b)$ as the length of the edge between them, then the length of $R$ is 
    $$length(R) = \sum_{i = 0}^{n-1} (r_i, r_{i+1}) $$
    First, notice that by the definition of path $r_i$ and $r_{i+1}$ are adjacent in $G$, then $\phi(r_i), \phi(r_{i+1})$ are adjacent in $H$, for every $i$, $0 \leq i \leq n-1$. Also by definition of $R$, $r_i\neq r_j$ for $i\neq j$, and by definition of $\phi$ this implies that $\phi(r_i) \neq \phi(r_j)$ for $i \neq j$.\\
    Then $Q = \{\phi(u) = \phi(r_0), \phi(r_1), ..., \phi(r_{n-1}), \phi(v) = \phi(r_{n})\}$ is a sequence of adjacent vertices without repetition in $H$, which implies that $Q$ is a $\phi(u)-\phi(v)$ path, and since $H$ is a tree, $Q$ is the unique  $\phi(u)-\phi(v)$ path in $H$.\\
    Finally, by definition of $\phi$ for every edge, $a-b$ in $G$, the edge $\phi(a)-\phi(b)$ in $H$ has the same weight, therefore:
    $$length(Q) = \sum_{i = 0}^{n-1} (\phi(r_i), \phi(r_{i+1})) = \sum_{i = 0}^{n-1} (r_i, r_{i+1}) = length(R)$$
    \end{proof}
    
    \begin{corollary}\label{cor:same_radium}
    For $G$ and $H$ two isomorphic trees, their radius will have the same length.
    \end{corollary}

    \begin{lemma}\label{lema:bij_neighbors}
    Let $G$ and $H$ be two isomorphic trees with isomorphism function $\phi$, for any $v$ vertex of $G$, the function $\phi \mid_{N(v)}$, where $N(v)$ is the set of the neighbors of $v$, is a bijective function with codomain $N(\phi(v))$ in $H$.
    \end{lemma}
    \begin{proof}
    Given $v$ a vertex of $G$, let $x \in N(v)$. That happens if and only if $length_1(v, x) = 1$. In the other hand, $\phi(x)$ is neighbor of $\phi(v)$ if and only if $length_1(\phi(v), \phi(x)) = 1$. By using a particular case of the previous lemma we have that $length_1(v, x) = length_1(\phi(v), \phi(x))$, therefore $x \in N(v)$ if and only if $\phi(x) \in N(\phi(v))$. Since $\phi$ is a bijection then $\phi \mid_{N(v)}: N(v)  \to N(\phi(v))$ is a bijection too. This means that the images of the neighbors of $v$ are the neighbors of the image of $v$. 
    \end{proof}

Let's denote $(a,b)_p$ as the path from $a$ to $b$ in the graph and $length_{1}(a, b)$ as its length (without weights), since all the graphs generated by the simply-connected trivalent 2-dimensional stratifolds are trees, by Theorem \ref{thm:unique_path} this is well defined. \\

\begin{lemma}\label{lemma:leaf_ecc}
For any vertex $v$ in a tree $G$, if $u$ is a vertex such that $length_1(u,v) = e(v)$ then $u$ is a leaf.
\end{lemma}
\begin{proof}
    Suppose that $u$ is not a leaf, then the degree of $u$ is at least $2$, then there exists $x$ neighbor of $u$ such that $x \not \in (v,u)_p$.  By theorem \ref{thm:unique_path} we know that such $x$ exists, because there is a unique path from $v$ to $u$ which means that there is only one neighbor of $u$ connected to $v$. If there where $x, y$ neighbors of $u$ connected to $v$, there would be two paths from $v$ to $u$ and that is a contradiction.\\
    Let's notice that the
    $$e(v) = length_1(v, u) <length_1(v, u) + 1 = length_1(v, u) + length_1(u, x) = length_1(v, x),$$ 
    but that's a contradiction to the definition of eccentricity. Therefore the degree of $u$ must be at most 1, and $u$ is a leaf of $G$.
\end{proof}

\begin{theorem}[Uniqueness of the center]\label{thm:center}
Let $G$ be a trivalent graph, then there exists $c$ center of $G$ and it is unique.
\end{theorem}
\begin{proof}
    By definition, there exists at least one vertex $v$ such that $e(v) = rad(G)$.\\
    First, let's notice that two vertices of different colors in a trivalent graph can't be both central vertices. Every trivalent graph has the property that all the neighbors of a white vertex are black vertices and vice versa, also all the leaves are white. Then by parity, the $length_1$ from any white vertex to a leaf would be even and the $length_1$ from any black vertex to a leaf would be odd, using the lemma \ref{lemma:leaf_ecc} we can conclude that the eccentricity of a white vertex would be always different from the eccentricity of a black vertex, which means that they can't be both central vertices.\\
    Now suppose that there exists $u \neq v$ such that $u$ is a central vertex of $G$. Without loss of generality we can assume that $u, v$ are both white. Then there exists $x$ a black vertex such that $x \in (u,v)_p$. We would prove that $e(x) < e(v)$. \\
    Let $l_1, l_2, ...., l_n$ be the leaves of $G$, by definition, for any leaf $l_i \in G$, we will have $length_1(v, l_i) \leq e(v)$ and $length_1(u, l_i) \leq e(u)$. \\
    Notice that for any $l_i$ we have two options, $x  \not \in (v,l_i)_p$ or  $x \in (v,l_i)_p$. We will analyze both cases.\\
    Case 1,  $x \not \in (v,l_i)_p$:\\
    If $x \not \in (v, l_i)_p$, in particular $(x, u)_p \not \subset (v, l_i)_p$ because $x \in (x,u)_p \subset (v, u)_p$ and the last one is unique, by lemma \ref{thm:unique_path}. Notice that there exists a unique  vertex $w \in (v, x)_p$ (it could happen that $w =v$) such that $(v, l_i)_p = (v, w)_p \cup (w, l_i)_p$ and $(w, l_i)_p \cap (v, x)_p = \{w\}$, therefore we have that\\
    $$length_1(w, l_i) + length_1(x, w) = length_1(x, l_i) < length_1(x, l_i) + length_1(u, x) = length_1(u, l_i) \leq e(u) = e(v)$$
    Then $length_1(x, l_i) < e(v)$. \\
    Case 2, $x \in (v, l_i)_p$:\\
    Since $x \in (v,l_i)_p$ then $(v, l_i)_p = (v, x)_p \cup (x, l_i)_p$ therefore\\
    $$length_1(x, l_i) < length_1(x, l_i) + length_1(v, x) = length_1(v, l_i) \leq e(v)$$ 
    which implies $length_1(x, l_i) < e(v)$.\\
    From both cases, we can conclude that $length_1(x, l_i) <e(v)$ for any leaf $l_i \in G$, then $$\displaystyle e(x) = \max_{1\leq i \leq n}length_1(x, l_i) < e(v),$$ 
    contradicting the fact that $v$ and $u$ were both central vertices because there is another vertex with lower eccentricity than the radius. \\
    This proves that for any trivalent graph, the center exists and it is unique.
\end{proof}

\begin{lemma}\label{iso_iff_isotg}
    Any two trivalent graphs are isomorphic as trivalent graphs if and only if they are isomorphic as graphs.
    \end{lemma}
    \begin{proof}
    By definition, it is clear that isomorphism as trivalent graphs implies isomorphism. We only need to prove that isomorphism implies isomorphism as trivalent graphs, which is the isomorphism function determines a bijection between the set of black vertices of both graphs, analogously with the white vertices. \\
    Let $G$ and $H$ be two trivalent graphs. If instead of the length in lemma \ref{lemma:same_length_paths} we consider the length (without weights) we can conclude that for any vertex $v \in G$, if $\phi(v) = w\in H$ then $e(v) = e(w)$.\\
    Using the same parity argument as in the previous proof, notice that $e(v) = e(w)$ if only if $v, w$ are both white or both black. Moreover, $e(v), e(w)$ would be odd if and only if $v$ and $w$ are black, and would be even if and only if $v$ and $w$ are white. \\
    This gives us a partition of the graph that assures us that the image of any white vertex is going to be a white vertex, and the image of any black vertex is going to be a black vertex. Therefore the restriction of $\phi$ to the set of black vertices in $G$ is a bijection with codomain the set of black vertices in $H$, analogously with the white vertices. Then we can conclude that isomorphism implies isomorphism as trivalent graphs.
    \end{proof}
    
\begin{lemma}\label{lemma:center_iso}
 For $G$ and $H$ two isomorphic trivalent graphs with isomorphism function $\phi$, then the center of $H$ is the image of the center of $G$ under $\phi$.
\end{lemma}
\begin{proof}
As a result of the lemma \ref{lemma:same_length_paths}, if we consider the length (without weights) instead of the length, for any $v \in G$, the eccentricity of $v$ in $G$ will be the same as the eccentricity of $\phi(v)$ in $H$. Let $c$ be the center of $G$, which is unique, using corollary \ref{cor:same_radium} we have that 
$$rad(H) = rad(G) = e(v) = e(\phi(c))$$
By definition, since $rad(H) = e(\phi(c))$ we can conclude that $\phi(c)$ is the center of $H$.
\end{proof}

\begin{lemma}\label{lemma:isotg_iff_isort}
Given $G$ and $H$ trivalent graphs, if we select the center of each graph as its root, $G$ and $H$ are isomorphic as trivalent graphs if and only if $G$ and $H$ are isomorphic as rooted trees.
\end{lemma}
\begin{proof}
If $G$ and $H$ are isomorphic as trivalent graphs, by lemma \ref{lemma:center_iso}, since the center of $H$ is the image of $G$ under the isomorphism function, it is immediate that $G$ and $H$ are isomorphic as rooted trees. Now let's suppose that $G$ and $H$ are isomorphic as rooted trees. By definition, the isomorphism as rooted trees implies that $G$ and $H$ are isomorphic, therefore by lemma \ref{lemma:isotg_iff_isort}, since $G$ and $H$ are two trivalent graphs that are isomorphic, then they are isomorphic as trivalent graphs. 
\end{proof}

\section{Characterizing weighted trees with a string}

So far, we have described how two isomorphic trivalent graphs behave, but we need tools to identify if two trivalent graphs are isomorphic. This is important because the main goal is to know how many and which are all the trivalent stratifolds for a given number of white vertices. Since the trivalent stratifolds are associated with a unique trivalent graph, having a classification for the trivalent graphs gives us  a classification for the trivalent stratifolds. \\

The generation of all the trivalent graphs is an iterative process that creates a lot of isomorphic graphs, we will discuss this process further in Section \ref{sec:labels}. But it is because of this excess of repetitions, that we need an optimal algorithm that can recognize isomorphic graphs with as few operations as possible.

In the book \textit{The Design and Analysis of Computer Algorithms} \cite{book:AHUBook} the authors propose an algorithm that allows us to identify if two non-weighted rooted trees are isomorphic in $O(n)$ time. This algorithm is known as the AHU algorithm, the acronyms AHU comes from the initials of the authors Aho, Hopcroft, and Ullman. To use this algorithm is necessary to remark that two isomorphic trees could be non-isomorphic as rooted trees (see Fig. \ref{fig:rooted_trees} for an example). \\

\begin{figure}[h]
    \centering
    \includegraphics[height = 3cm]{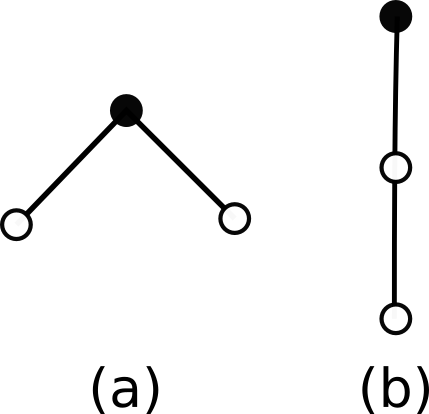}
    \caption{This is an example of two isomorphic trees that aren't isomorphic as rooted trees. On each tree, we have marked in bold black the root. }
    \label{fig:rooted_trees}
\end{figure}

In the article \textit{Tree isomorphism Algorithms: Speed vs. Clarity} \cite{art:AHUalgorithm}, there is an algorithm that improves the idea of Aho, Hopcroft, and Ullman, by implementing the use of \textit{parenthetical tuples}. And then substituting the use of `(', `)' for `1' and `0' respectively,  the latest have a natural order.

Given a rooted tree $T$, the main idea of the algorithm is to assign a unique string to each vertex of $T$ recursively. The string assigned to a vertex is created recursively from the string associated with its children. And finally, assign to $T$ the string associated with its root. Then we can conclude that two rooted trees are isomorphic if and only if they have the same associated string.\\
 
We now present the pseudo-code of the AHU Algorithm. An example of this process can be seen in Fig.~\ref{fig:AHUAlgorithm}.

\begin{algorithm}[h]
\caption{AHU($v$: vertex)}
\begin{algorithmic}
\IF{$v$ is childless}
\STATE Give $v$ the tuple name ``10''
\RETURN ``10''
\ELSE
\STATE Set $L = \emptyset$  
\FORALL{$w$ child of $v$}
\STATE $tag$ = \texttt{AHU($w$)};
\STATE Append $tag$ to $L$
\ENDFOR
\STATE Sort $L$ using binary order
\STATE Set $temp = $ Concatenation of tags in $L$
\STATE Give $v$ the tuple name ``$1temp0$''
\ENDIF
\end{algorithmic}
\end{algorithm}

\begin{figure}
    \centering
    \includegraphics[height = 5cm]{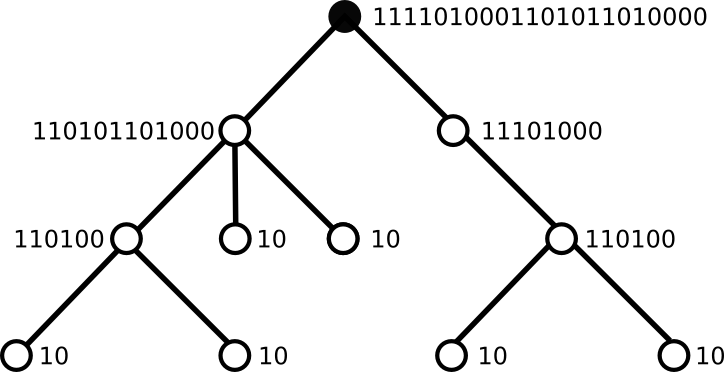}
    \caption{Example of labels given by running the AHU to a rooted tree}
    \label{fig:AHUAlgorithm}
\end{figure}

Although this algorithm allows us to recognize if two rooted trees are isomorphic, it doesn't allow us to distinguish if two rooted trees with weights are isomorphic, because it doesn't take into account the weights. Since trivalent graphs are weighted trees, this algorithm doesn't work for our problem in the first instance.\\

By theorem \ref{thm:center}, given a trivalent graph $G$, there is a unique vertex $v$ such that $v$ is the central vertex of $G$, we call $v$ the center of $G$. The uniqueness of the center for every trivalent graph allows us to mark this vertex as the root of the trivalent graph, without ambiguity. 

We have proven that the center exists, but so far we only have given an exhaustive algorithm to find it. In \cite{book:AHUBook} on pages 176 to 179, Aho, Hopcroft, and Ullman describe the algorithm \textit{Depth-first search} whose purpose is to find the largest path in a tree. It is proven that this algorithm only needs $O(\max\{n, e\})$ steps on a graph with $n$ vertices and $e$ edges. The idea of this algorithm is to visit every vertex of the tree in an ordered way, going deeper before continuing to another branch of the tree, this way we can assure that we visit every vertex exactly once.\\

\begin{algorithm}
\caption{Depth-first\_search($v$:vertex)}
\begin{algorithmic}
\IF{$v$ is childless} 
    \RETURN ${v}$
\ELSE 
    \STATE Set $length=0$ and $longestPath = \emptyset$
    \FORALL{$w$ child of $v$}
        \STATE Set $path=$\texttt{Depth-first\_search($w$)} and $L$ as the length of $path$.
        \IF{$L > length$ }
        \STATE Set $length= L$ and $longestPath = path$
        \ENDIF
    \ENDFOR
    \RETURN ${v} \cup longestPath$
\ENDIF
\end{algorithmic}
\end{algorithm}

To find the center of the trivalent graph it is only necessary to find the longest path in the tree and then find the middle vertex of it. This will always be the center of the tree. The pseudo-code is the algorithm \ref{alg:center-trivalent-graph}.

The Algorithm \ref{alg:center-trivalent-graph} successfully finds the center of any trivalent graph because it first finds a path of maximal length (a diameter) of the tree. Any diameter of a trivalent has an even length (see Theorem \ref{thm:center} ). By the definition of the center, all vertices must be at a distance at most half of the diameter. That is why the center must be in the middle of any diameter. 

\begin{algorithm}\label{alg:center-trivalent-graph}
\caption{center($G$: Trivalent graph)}
\begin{algorithmic}
    \STATE Set $v$ a vertex of $G$
    \STATE Set $longestPath =$\texttt{Depth-first\_search($v$)}
    \STATE Set $w$ as the last vertex of the path $longestPath$.
    \STATE Set $longestPath =$\texttt{Depth-first\_search($w$)}
    \STATE Set $center$ as the middle vertex of $longestPath$
    \STATE return $center$
\end{algorithmic}
\end{algorithm}

To prove that we can find a diameter with two runs of DFS we do as follows. What we have to show is that at the first run of the DFS we end up at the end of a diameter of $G$. So, at the next run, we will end up at the other end of the diameter. Then, what we have to show is that that farthest vertex $w$ from a given vertex $v$ in $G$ belongs to a diameter of $G$.

Let $w$ be one of the farthest vertices from $v$. Observe first that $w$ has to be a leaf, otherwise, we can find a farther vertex from $v$.  Now, let $\gamma$ a diametrical path of $G$. As $w$ is a leaf, if $w$ belongs to $\gamma$ then the proof would be over. Lets call $a$ and $b$ the ends of $\gamma$. By the definition of $w$, the vertices $a$ and $b$ need to be no further from $v$ than $w$. If one of them were closer to $v$, we can prove that one of the paths from $w$ to $a$ or $b$ is longer than $\gamma$ (an analysis by cases is needed it here). This contradicts the fact that $\gamma$ is a diameter. Similarly, if both $a$ and $b$ are as far from $v$ than $w$, it implies that one of the paths from $w$ to $a$ or $b$ is as long as $\gamma$, making $w$ the end of a diameter. Which completes the proof. So, the second time that the DFS runs, it will find the other end of the diameter.

This algorithm needs twice the number of steps that \texttt{Deep-first\_search} plus a constant, which means  that this algorithm is still linear and only depends on the number of vertices and edges of the graph.
 
Now, every trivalent graph can be seen as a rooted tree and therefore we can apply the AHU algorithm to it. The problem is that the AHU algorithm doesn't take into account the weights of the edges, but we have solved this problem by instead of assigning only numbers `1' and `0' we include the numbers `2' and `3' depending on the weight in the edge that connects the vertex with its father. This algorithm still runs on linear time.

The recursive part of the algorithm is given by the following pseudo-code:

\begin{algorithm}[h]
\caption{AHU-modified($v$:vertex)} \label{AHU_modified}
\begin{algorithmic}
\IF{$v$ is childless}
        \IF {$v$ has no father \OR  $Weight[v,father(v)] =1$ }
            \STATE Give $v$ the tuple name ``01'';
        \ELSE 
            \STATE Give $v$ the tuple name ``23'';
        \ENDIF
    \RETURN  The tuple name of $v$
\ELSE
\STATE Set $L = \emptyset$  
\FORALL{$w$ child of $v$}
\STATE $tag$ = \texttt{AHU-modified($w$)};
\STATE Append $tag$ to $L$
\ENDFOR
\STATE Sort $L$ using base four order
\STATE Set $temp = $ Concatenation of tags in $L$
    \IF {$v$ has no father \OR  $Weight[v,father(v)] =1$ }
        \STATE Give $v$ the tuple name ``$0temp1$'';
    \ELSE 
        \STATE Give $v$ the tuple name ``$2temp3$'';
    \ENDIF
\RETURN The tuple name of $v$
\ENDIF
\end{algorithmic}
\end{algorithm}

Now, given any trivalent graph, to get its string it is necessary to get its center first and then get the string associated with that vertex. For the complete implementation, see Algorithm \ref{TGtoString}.
\begin{algorithm}[ht]\label{TGtoString}
\caption{TG\_to\_string($G$: Trivalent graph)}
\begin{algorithmic}
    \STATE Set $c$ as the output of \texttt{center($G$)};
    \STATE Set $c$ as the root of $G$.
    \STATE \texttt{Vertex\_to\_string($c$)};
    \STATE Return the tuple name of $c$;
\end{algorithmic}
\end{algorithm}

\begin{definition}\label{def:string-representation-of-a-tree}
Given a trivalent graph $G$ we call the output of \texttt{TG\_to\_string} ($G$) [\ref{TGtoString}] as the \emph{string representation of $G$}. 
\end{definition}

\begin{theorem}\label{thm:String_iff_TG}
Given two trivalent graphs, they are isomorphic if and only if they have the same string representation.
\end{theorem}

To prove this theorem, we are going to give an algorithm that recovers the original graph given a string, and this will prove that there is a unique graph associated with any string.

\begin{algorithm}[H]
\label{alg:string-to-tg}
\caption{String\_to\_TG($S$: string, $father$: vertex)}
\begin{algorithmic}
    \IF{$father$ is \texttt{NONE}}
        \STATE Draw a vertex $v$;
        \STATE State $father$ as $v$;
    \ENDIF
    \IF{The first element of $S$ is 0}
        \STATE Set $Close$ as 1;
    \ELSE
        \STATE Set $Close$ as 3;
    \ENDIF 
    \STATE Set $i$ as 2
    \WHILE{The $i$-th element of $S$ is different from $Close$}
    \IF{The $i$-th element of $S$ is 0}
        \STATE Draw a vertex $w$ connected to $father$ with weight 1;
    \ELSE 
        \STATE Draw a vertex $w$ connected to $father$ with weight 2;
    \ENDIF
    \STATE Set $P$ as the string $S$ without its first element;
    \STATE Set $i$ = \texttt{String\_to\_TG($P$, $w$)} + 2;
    \ENDWHILE
    \STATE Return $i$;
\end{algorithmic}
\end{algorithm}

This algorithm draws a unique trivalent graph given a string representation, therefore we are giving a bijection between the trivalent graphs and the string representations, proving the theorem \ref{thm:String_iff_TG}. 

This algorithm can be extended for n-colored trees in general, also it can be extended for trees with a greater amount of weights, it only needed to add more start-close indicator numbers to identify the different weights.

\section{The graph generator algorithm}
Given a tivalent graph, one can generate more by applying the operations O1 or O2 to any white vertex of it. Or given two trivalent graphs, we can generate a new one by applying the operation O1* in one white vertex of each graph. It is proven in \cite{art:Models} that all the trivalent graphs can be obtained by recursively using these operations in all the previous trivalent graphs, starting with the B111 and B12 trees. 

The B111 and B12 trees are defined in \cite{art:Models} in Definition 1. 
\begin{definition}
\begin{enumerate}
    \item The B111-tree is the bi-colored tree consisting of one black vertex incident to three edges each of label 1 and three terminal white vertices each of genus 0.
    \item The B12-tree is the bi-colored tree consisting of one black vertex incident to two edges one of label 1, the other of label 2, and two terminal white vertices each of genus 0.
\end{enumerate}
\end{definition}

Also the operations O1, O2 and O1* are defined in \cite{art:Models}. 
\begin{definition}
In a trivalent graph $\Gamma$ let $w$ be a white vertex and let $e_1, ..., e_m$ be the edges incident to $w$ ($m \geq 0$) and let $b_i$ be the black vertex incident to $e_i$($i = 1, ..., m$). We define the operations $O1$ and $O2$ on $\Gamma$ that changes $\Gamma$ to a new trivalent graph $\Gamma_1$ as follows:
\begin{enumerate}
    \item \emph{O1.} Let $0 \leq k \leq m$. Attach one white vertex of a B111-tree to $w$, cut off $b_{k-1}, ..., b_m$ from $w$ and attach $b_{k+1}, ..., b_m$ to another white vertex of the B111-tree.
    \item \emph{O2.} Attach a B12-tree to $w$ so that the terminal edge has label 1.
    \item \emph{O1*.} On the other hand, let $\Gamma_1$ and $\Gamma_2$ be two disjoint trivalent graphs and let $w_i$ be a white vertex of $\Gamma_i$ ($i = 1,2$). Attach a B111-tree to $\Gamma_1 \cup \Gamma_2$ so that $w_1$ and $w_2$ are identified with two distinct white vertices of the B111-tree. 
\end{enumerate}
\end{definition}

To get all the trivalent graphs with $n$ white vertices, we have implemented a program that has two fundamental parts. The first part constructs all the trivalent graphs with $i$ white vertices ($2\leq i \leq n$) and the second part reduces the list by eliminating the repetitions. 

\begin{algorithm}\label{alg:Constructor}
\caption{Construct\_TG($m$:integer)}
\begin{algorithmic}
    \STATE Create $Complete\_list$ a list with $m-1$ empty lists;
    \STATE Set the B12-tree as the first element of $Complete\_list[0]$; 
    \STATE Set the B111-tree as the first element of $Complete\_list[1]$; 
    \FORALL{$w$ white vertex of B12-tree}
        \STATE Set $\Gamma$ as the rif the graph is already there. Here is when we have to use the Characterizing string from ing O2 to $w$;
        \STATE Add $\Gamma$ to $Complete\_list[1]$;
    \ENDFOR
    \FOR{$n$ in $[4, n]$}
        \FORALL{$q$ graph in $Complete\_list[n-3]$}
            \FORALL{$w$ white vertex of $g$}
                \STATE Set $\Gamma$ as the result of applying O2 to $g$ in the vertex $w$.
                \STATE Add $\Gamma$ to $Complete\_list[n-1]$ if it not already there.
            \ENDFOR
        \ENDFOR
        
        \FORALL{$g$ graph in $Complete\_list[n-4]$}
            \FORALL{$w$ white vertex of $g$}
                \STATE Set $\Gamma$ as the result of applying O1 to $g$ in the vertex $w$.
                \STATE Add $\Gamma$ to $Complete\_list[n-1]$ if it not already there.
            \ENDFOR
        \ENDFOR
        \FOR{$i$ in $[0, n-1]$}
            \IF{$n-i-5 \geq 0$}
                \FORALL{pair $(u, v)$ where $u$ is a white vertex of a graph in $Complete\_list[i]$, $v$ is a white vertex of a graph in $Complete\_list[n-i-5]$}
                    \STATE Set $\Gamma$ as the output of applying $O1^*$ using the vertices $u, v$;
                    \STATE Add $\Gamma$ to $Complete\_list[n-1]$ if it not already there.
                \ENDFOR
            \ENDIF
        \ENDFOR
    \ENDFOR
\end{algorithmic}
\end{algorithm}

We should mention that whenever we add a graph to $Complete\_list[n-1]$ we have to check if the graph is already there. Here is when we have to use the string representation of its elements (see Def. \ref{def:string-representation-of-a-tree}) to compare them. Even more, we can set the string representation as a hashing function, so we don't have to iterate over all the elements of $Complete\_list[n-1]$ to decide if it already there or not, and decide it in amortized constant time. The complete implementation can be found in the repository \url{https://github.com/MyHerket/TrivalentStratifold} in GitHub.

  As explained before, the creation of a new trivalent graph is an iterative process. And we are going to prove that the previous algorithm creates all the trivalent graphs with $n$ white vertices. \\
    
    \begin{lemma}\label{one_leaf}
    Let $G$ be a trivalent graph, there exists at least one leaf $w$  of $G$ such that the weight of the adjacent edge to $w$ is 1.
     \end{lemma}
     \begin{proof}
    We are going to proceed by induction. First, notice that for $b12-$ and $b111-tree$ there exists $w$ a leaf such that the weight of the adjacent edge to $w$ is 1.\\
    Let $G$ be a trivalent graph, if we perform $O1$ in one of its vertices, we are attaching a $b111-tree$ by one of its white vertices, letting 2 white vertices free that are going to be leaves of $G$ with such that the weight of the adjacent edges to them is 1.\\
    On the other hand, if we perform $O2$ in one of $G$'s vertices, we are attaching a $b12-tree$ to it by the only white vertex whose adjacent edge weight is 2, therefore the white vertex whose adjacent edge weight is 1, is now a leaf of the new graph\\
    Finally, if we perform $O1^*$ to $G$ and other graph, by definition we take a $b111-tree$ and attach one white vertex to $G$, one to the other graph and the last one is free, which is the leaf whose adjacent edge weight is 1.\\
    Since the process to get any trivalent graph is taking $b12$ or $b111$ and performing $O1, O2$ or $O1^*$, as many times as we want in each step we have a leaf whose adjacent edge weight is 1, therefore the resulting trivalent graph has it.
     \end{proof}
     \begin{remark}
    Let $G$ be a trivalent graph with $k$ white vertices, the resulting trivalent graph $G'$ after performing $O1$ in any white vertex of $G$ has $k+2$ white vertices.
    \end{remark}
    \begin{remark}
    Let $G$ be a trivalent graph with $k$ white vertices, the resulting trivalent graph $G'$ after performing $O2$ in any white vertex of $G$ has $k+1$ white vertices.
    \end{remark}
    \begin{remark}
    Let $G$ and $G'$ be trivalent graphs with $k$ and $j$ white vertices, respectively. The resulting graph $H$ after performing $O1^*$ in any pair of white vertices of $G$ and $G'$ has $k+j+1$ white nodes.
    \end{remark}
    Using the previous remarks we have the following theorem:
    \begin{theorem}\label{algorithm}
    For any $n$ an integer greater than 3. The list of all the trivalent graphs with $n$ white vertices, (including isomorphisms) will be obtained by performing the operation $O1$ to all the trivalent graphs with $n-2$ white vertices in each of their white vertices, performing the operation $O2$ to all the trivalent graphs with $n-1$ white nodes in each of their vertices and finally for every $m$, such that $2\leq m \leq n-3$ perform the operation $O1^*$ in all the pairs of white vertices of every pair of trivalent graphs where the first one is an element of the list of trivalent graphs with $m$ white vertices and the second one is an element of the list of trivalent graphs with $n-m-1$ white vertices.
    \end{theorem}
    \begin{proof}
    Let $n$ be an integer greater than 3. By the remarks, it is clear that performing the algorithm described will give us a subset of the list of all the trivalent graphs with $n$ white vertices. Let's see that this subset is the total set.\\
    Let $G$ be a trivalent graph with $n > 3$ white vertices and $w$ a leaf of $G$ such that the edge adjacent to $w$ has weight 1, it exists by lemma \ref{one_leaf}. We know that $w$ is a white vertex, by theorem 1 in \cite{art:Models}. 
    Let $b$ be the black node adjacent to $w$. If $b$ has degree 2, let $v$ be the other vertex adjacent to $b$, when we erase the vertices $w, b$ we get a new graph $G'$ with $n-1$ white vertices such that after performing $O2$ in $G'$ in the vertex $v$ we get $G$.\\
    If $b$ has degree 3, then $b$ is part of a $b111-sub tree$, and let $v_1$ and $v_2$ the vertices adjacent to $b$ different to $w$. If there is $v_i$ ($ i \in {1, 2}$) such that its degree is one, suppose $v_1$, when we erase the vertices $w, b, v_1$, we get a new trivalent graph $G'$ with $n-2$ white vertices such that after performing $O1$ in $v_2$ we get the original trivalent graph $G$. On the other hand when neither $v_1$ nor $v_2$ has degree 1, when we erase the vertices $b, w$ we get two trivalent graphs $H$ and $H'$ such that the sum of their white vertices is $n-1$ and after performing $O1^*$ in $v_1, v_2$ we get the original graph $G$.\\
    Therefore every trivalent graph of the list was a result of a step in the algorithm described in the theorem and the algorithm gives us all the graphs, including isomorphisms.\\
    \end{proof}
  
\begin{table}[ht]
    \centering
     \begin{tabular}{||c| c | c||} 
     \hline
     $n$ & Total & Created \\ [0.5ex] 
     \hline\hline
    2 &    1      & 1\\
    3 &    3      & 3\\
    4 &    6      & 11\\
    5 &    18	  & 37\\
    6 &    51	  & 150\\
    7 &    167	  & 573\\
    8 &    551    & 2267\\
    9 &    1954	  & 8997\\
    10&    7066	  & 36498\\
    11&    26486  & 149708\\
     [1ex] 
     \hline
    \end{tabular}
    \caption{Number of distinct graphs we got for each value $n$ (the number of white vertices), and the number of graphs that were created to construct them all.}
    \label{tab:my_label}
\end{table}

This process is exhaustive, we can assure that it creates all the trivalent graphs with $n$ white vertices, but it creates too many repetitions. This is because there are symmetries in the rooted trivalent graphs, and when applying any operation in two symmetric vertices the resulting trivalent graph is the same.\\

\begin{definition}\label{def:symmetry}
Given $G$ a rooted trivalent graph, we say that two vertices $u, v \in G$ are \emph{symmetric} if there exists an automorphism $\phi: G \to G$ (as rooted weighted graphs) such that $\phi(u) = v$.
\end{definition}

It is clear that if two vertices are symmetric, they must have the same string representation because the string representation of a vertex $v$ depends solely on the isomorphisms class of the subtree defined by the descendants of $v$. Unfortunately, having the same string representation is not enough to recognize symmetric vertices. The fathers of two symmetrical vertices $u$ and $v$ must be symmetrical as well ( an automorphism $\phi: G \to G$ must send the father of $u$ to the father of $v$). This implies, that the process of detecting symmetrical vertices can be iterated recursively, and it will end when the fathers of two vertices coincide (when they are siblings). 

We use the above idea for detecting symmetric white vertices of a graph $G$. And we modified algorithm \ref{alg:Constructor} to work only with symmetrically distinct white vertices.  These changes reduced the number of generated graphs by around $20\%$ as shown in Table \ref{tab:created-graphs-after-symmetry}.

\begin{table}[ht]
    \centering
     \begin{tabular}{||c| c | c||} 
     \hline
     $n$ & Created & Reduction \\ [0.5ex] 
     \hline\hline
    4 &  11     &  0,00\% \\
    5 & 32      &  13,51\% \\
    6 & 122     &  18,67\% \\
    7 & 467     &  18,50\% \\
    8 &  1781   &  21,44\% \\
    9  & 7099   &  21,10\% \\
    10 & 28852  &  20,95\% \\
    11 & 119168 &  20,40\% \\
     [1ex] 
     \hline
    \end{tabular}
    \caption{Number of created graphs after considering symmetrically distinct white vertices.}
    \label{tab:created-graphs-after-symmetry}
\end{table}

\begin{center}

\end{center}

More optimization could've been done to avoid so many repetitions, but the program would've run in exponential time anyway. Because the number of graphs we want to construct has exponential growth.\\

\section{Nomenclature}\label{sec:labels}
    We need a nomenclature to differentiate the trivalent graphs. For $G$ a trivalent graph, the tag is going to be the identifier of $G$. To have a general idea of the shape of the graph it is necessary to include the number of leaves, black and white nodes. Also, we will include the length of the largest and shortest leaf paths of $G$. And an ID number that identifies $G$ as a unique graph.\\
    Denote $W(G), B(G), L(G)$ the sets of white vertices, black vertices and leaves of $G$, respectively.\\
    The tag will have the following structure:
    \begin{center}
        $tag(G)$ = [ \textit{$|W(G)|$, $|B(G)|$, $|L(G)|$, length(shortest leaf path), length(largest leaf path), ID number}]
    \end{center}
    The $ID$ number is automatically generated by the program when using the Hash table that uses the unique string representation of $G$ to order it in the list of trivalent graphs. 

\bibliographystyle{plain}
\bibliography{ref}

\begin{thebibliography}{1}

\bibitem{cod:TSR}
Trivalent stratifold repository.

\bibitem{book:AHUBook}
Alfred~V. Aho, John~E. Hopcroft, and Jeffrey~D. Ullman.
\newblock {\em The design and analysis of computer algorithms}.
\newblock Addison-Wesley Publishing Co., Reading, Mass.-London-Amsterdam, 1975.
\newblock Second printing, Addison-Wesley Series in Computer Science and
  Information Processing.

\bibitem{art:AHUalgorithm}
Douglas~M. Campbell and David Radford.
\newblock Tree isomorphism algorithms: speed vs. clarity.
\newblock {\em Math. Mag.}, 64(4):252--261, 1991.

\bibitem{book:GraphT}
Gary Chartrand and Ping Zhang.
\newblock {\em Chromatic graph theory}.
\newblock Discrete Mathematics and its Applications (Boca Raton). CRC Press,
  Boca Raton, FL, 2009.

\bibitem{art:2-stratifold}
J.~C. G\'{o}mez-Larra\~{n}aga, F.~Gonz\'{a}lez-Acu\~{n}a, and Wolfgang Heil.
\newblock 2-dimensional stratifolds.
\newblock In {\em A mathematical tribute to {P}rofessor {J}os\'{e} {M}ar\'{\i}a
  {M}ontesinos {A}milibia}, pages 395--405. Dep. Geom. Topol. Fac. Cien. Mat.
  UCM, Madrid, 2016.

\bibitem{art:class}
J.~C. Gomez-Larra\~{n}aga, F.~Gonz\'{a}lez-Acu\~{n}a, and Wolfgang Heil.
\newblock Classification of simply-connected trivalent 2-dimensional
  stratifolds.
\newblock {\em Topology Proc.}, 52:329--340, 2018.

\bibitem{art:Models}
J.~C. Gomez-Larra\~{n}aga, F.~Gonz\'{a}lez-Acu\~{n}a, and Wolfgang Heil.
\newblock Models of simply-connected trivalent 2-dimensional stratifolds.
\newblock {\em Bol. Soc. Mat. Mex.}, 26:1301–--1312, 2020.

\bibitem{StratModule}
Hernández Yair.
\newblock Stratifolds.
\newblock https://github.com/yair-hdz/stratifolds, 2018.

\end{thebibliography}

\end{document}